\documentclass[10pt]{amsart}

\usepackage{amssymb, tikz}

\usepackage{hyperref}

\usepackage{amsmath}

\newtheorem{theorem}{Theorem}[section]

\newtheorem{lemma}[theorem]{Lemma}

\newtheorem{question}[theorem]{Question}
\newtheorem{proposition}[theorem]{Proposition}
\newtheorem{corollary}[theorem]{Corollary}
\newtheorem{claim}[theorem]{Claim}

\theoremstyle{definition}
\newtheorem{definition}[theorem]{Definition}

\newcommand{\ga}{\alpha}

\renewcommand{\gg}{\gamma}
\newcommand{\gd}{\delta}
\newcommand{\gw}{\omega}

\newcommand{\supp}{\mathrm{supp}}
\newcommand{\dom}{\mathrm{dom}}
\renewcommand{\restriction}{\mathord{\upharpoonright}}

\title{On the consistency strength of $\textsf{MM}(\omega_1)$}

\author[Dobrinen]{Natasha Dobrinen}
\address{Natasha Dobrinen, Department of Mathematics, University of Notre Dame, Notre Dame, IN 46556, USA}

\email{ndobrine@nd.edu}

\author[Krueger]{John Krueger}
\address{John Krueger, Department of Mathematics, University of North Texas, Denton, TX 76203, USA}

\email{jkrueger@unt.edu}

\author[Marun]{Pedro Marun}
\address{Pedro Marun, Department of Mathematical Sciences, Carnegie Mellon University,
Pittsburgh, PA, 15213, USA}

\email{pmarun@andrew.cmu.edu}

\author[Mota]{Miguel Angel Mota}

\address{Miguel Angel Mota,  Departamento de Matem\'aticas,
ITAM,
01080, Mexico City, Mexico}

\email{motagaytan@gmail.com}

\author[Zapletal]{Jindrich Zapletal}
\address{Jindrich Zapletal, Department of Mathematics, University of Florida, Gainesville, FL 32611, USA}

\email{zapletal@ufl.edu}

\thanks{This paper was originally conceived during the workshop \emph{From $\aleph_2$ to Infinity} sponsored by the American Institute of Mathematics and the National Science Foundation. We thank James Cummings, Itay Neeman, and Dima Sinapova for having organized that meeting which took place on May 15-19, 2023 at the AIM Institute in San Jose, California.}

\thanks{Natasha Dobrinen was supported by the National Science Foundation under grant NSF DMS-1901753. John Krueger was supported by the Simons Foundation under Grant 631279. }
\date{}

\begin{document}


\begin{abstract}
We prove that the consistency of Martin's Maximum restricted to partial orders of cardinality $\omega_1$ follows from the consistency of $\textsf{ZFC}$.
\end{abstract}

\maketitle
\pagestyle{myheadings}\markright{On the consistency strength of $\textsf{MM}(\omega_1)$}

\section{Introduction}

Given the profusion of independence results which followed Cohen's discovery of the method of forcing, it has become a major objective of set theory to find natural axiomatic extensions of \textsf{ZFC} which decide Cantor's Continuum Problem as well as other important questions undecidable in \textsf{ZFC}. For example, in the last five decades forcing axioms have been widely studied and shown to have very interesting consequences regarding the continuum. Intuitively, the idea behind them is that the universe of set theory must be somehow saturated under forcing. More precisely, given a class $\Gamma$ of partial orders and a cardinal $\kappa$, the \emph{forcing axiom for $\Gamma$ and $\kappa$}, $\textsf{FA}(\Gamma, \kappa)$, is the assertion that for every $P\in\Gamma$ and every collection $\mathcal D$ of size at most $\kappa$ consisting of
dense subsets of $P$, there is a filter $G\subset P$ such that $G\cap D\neq\emptyset$ for every $D\in\mathcal D$.

Some forcing axioms which are especially significant for their wide range of applications not only in set theory, but also in algebra, analysis, topology, and other fields, are Martin's Axiom for $\omega_1$-many dense sets ($\textsf{MA}_{\omega_1}$)
introduced by Martin, Solovay and Tennenbaum in the mid 1960's, the Proper Forcing Axiom ($\textsf{PFA}$) introduced by Baumgartner and Shelah in the early 1980's, and the Semiproper Forcing Axiom ($\textsf{SPFA}$) and Martin's Maximum ($\textsf{MM}$) introduced by Foreman, Magidor and Shelah in the mid 1980's. They are defined as $\textsf{FA}(\Gamma, \omega_1)$ for $\Gamma$ being, respectively, the class of all posets with the countable chain condition, the class of all proper posets, the class of all semiproper posets, and the class of all posets preserving stationary subsets of $\omega_1$, where these four classes are being presented in increasing order.

The forcing axioms $\textsf{MA}_{\omega_1}$, $\textsf{PFA}$ and $\textsf{SPFA}$ are known to be relatively consistent (from  \textsf{ZFC} in the first case and modulo large cardinals in the other two) by means of forcing iterations which fall in the same class $\Gamma$ being considered. So, this kind of construction  heavily depends on certain preservation criteria. One of them is the central theorem of Shelah stating that if $P_{\alpha}$ is a countable support forcing iteration of $\{\dot{Q}_\beta \colon \beta< \alpha\}$ such that every $\dot{Q}_\beta$ is a proper forcing notion in $V^{P_\alpha\restriction \beta}$, then $P_{\alpha}$ is proper
(in particular, $P_{\alpha}$ does not collapse $\omega_1$). Another one, also due to Shelah, holds in the context of revised countable support forcing iterations and semiproper forcings. On the other hand, there is no such preservation result for stationary set preserving posets and the classical argument for the consistency of $\textsf{MM}$ (see \cite{FMS} and \cite{SEMIPROPER}) goes in a slightly different way: it passes by showing that $\textsf{SPFA}$ implies that every stationary set preserving notion of forcing is semiproper, which in turn implies the equivalence between $\textsf{SPFA}$ and $\textsf{MM}$.

Let us denote by $\textsf{PFA}(\omega_1)$ and
$\textsf{MM}(\omega_1)$ the respective restrictions of 
$\textsf{PFA}$ and $\textsf{MM}$  to posets of cardinality $\omega_1$.
So, $\textsf{PFA}(\omega_1)$ and $\textsf{MM}(\omega_1)$
are defined as $\textsf{FA}(\Gamma, \omega_1)$, for $\Gamma$ being the class of all posets of cardinality $\omega_1$ which are proper in the first case,
and the class of all posets of cardinality $\omega_1$
which preserve stationary subsets of $\omega_1$ in the second case.
It is well-known that $\textsf{ZFC}$ and $\textsf{ZFC}+\textsf{PFA}(\omega_1)$ are equiconsistent, which follows from the fact that under \textsf{CH},
forcings of size $\omega_1$
can be iterated with countable support
up to length $\omega_2$ with an $\omega_2$-c.c.\ forcing iteration
(see Lemmas 2.4 and 2.5 of Chapter VIII of \cite{PROPER}). Shelah proved that $\textsf{ZFC} \, +$ ``there exists a strongly inaccessible cardinal''
implies the consistency of $\textsf{ZFC}+\textsf{MM}(\omega_1)$ (see Theorem 4.3 of Chapter III
of \cite{PROPER}).
The main theorem of this article states that Shelah's inaccessible can be removed from
this consistency result.

\begin{theorem}\label{mainthm-0-intro}
Assume $\textsf{CH}$ and  $2^{\omega_1}=\omega_2$. Then there is a countable support forcing iteration $P_{\omega_2}$ of $\{\dot{Q}_\beta \colon \beta< \omega_2\}$ with the following properties:
\begin{enumerate}
\item Every $\dot{Q}_\beta$ is, in $V^{P_{\omega_2}\restriction \beta}$, a proper poset;
\item $P_{\omega_2}$ is proper and has the $\omega_2$-chain condition;
\item $P_{\omega_2}$ forces $\textsf{MM}(\omega_1)$.
\end{enumerate}
Consequently, the theories $\textsf{ZFC}$ and $\textsf{ZFC}+\textsf{MM}(\omega_1)$ are 
equiconsistent.
\end{theorem}

We would like to thank the referee for pointing out an error in an earlier 
version of this article.

\section{Stationary Set Preserving But Not Proper}

Before proving the main theorem,
we give a brief sketch of the consistency that there exists a
forcing poset of size $\omega_1$ which preserves stationary subsets of
$\omega_1$ but is not proper.
This fact means that the work done in the next section for destroying stationary
set preserving posets which are not proper is not vacuous.
Another proof was given previously
by Sakai \cite{Sakai}, who introduced a new combinatorial principle called
$\diamondsuit^{++}$ and showed that it
implies the existence of a non-proper poset of size $\omega_1$ preserving stationary subsets
of $\omega_1$.
Since $\diamondsuit^{++}$ is forceable and holds in $L$, the existence of such a poset is 
consistent with $\textsf{ZFC}$.
Similar to the construction in the present section, Sakai's forcing is actually a Kurepa tree,
although it is not presented as such.

It is worth pointing out in this context that a poset of size $\omega_1$ is proper
iff it is semiproper.
Namely, if $Q$ is a poset of size $\omega_1$ and $N$ is a countable elementary submodel,
then by a straightforward argument any condition in $Q$ is a master condition for $N$
iff it is a semi-master condition for $N$.
Thus, we are actually proving the consistency of the
existence of a poset of size $\omega_1$
which preserves stationary sets and is not semiproper.

The forcing for introducing such a poset is the standard forcing for adding
a Kurepa tree, and the poset which is stationary set preserving but not proper
in the generic extension is the generic Kurepa tree.
For any ordinal $\alpha$,
let $P(\alpha)$ be the poset consisting of conditions which are pairs $(s,f)$, where
$s$ is a countable tree of height a countable successor ordinal satisfying
the usual normality properties, and $f$ is a function whose domain is a countable
subset of $\alpha$ so that for all $\gamma \in \dom(f)$, $f(\gamma)$ is
an element of the top level of $s$.
Let $(t,g) \le (s,f)$ if $t$ end-extends $s$, $\dom(f) \subset \dom(g)$,
and for all $\gamma \in \dom(f)$, $f(\gamma) \le_t g(\gamma)$.

The basic properties of $P(\alpha)$ are as follows.
\begin{enumerate}
	\item $P(\alpha)$ is countably closed and, assuming \textsf{CH}, $\omega_2$-c.c.
	\item Assume that $G$ is a generic filter on $P(\alpha)$.
	The union of the trees $s$, where $(s,f) \in G$ for some $f$, is a normal
	$\omega_1$-tree which we will denote by $T$, and the canonical name by $\dot T$.
	\item For each $i < \alpha$,
	the downwards closure
	of $\{ f(i) : \exists s \ (s,f) \in G, \ i \in \dom(f) \}$
	is a cofinal branch of $T$, which we will denote by $b_i$, and the canonical
	name by $\dot b_i$.
	\item For all distinct $i, j < \alpha$, the set of $(s,f) \in P(\alpha)$ such that
	$i, j \in \dom(f)$ and $f(i) \ne f(j)$ is dense open.
	Hence, $b_i \ne b_j$ in $V[G]$.
	In particular, if $\alpha \ge \omega_2$ then $T$ is a Kurepa tree in $V[G]$.
\end{enumerate}

We also note that if $\alpha < \beta$ then $P(\alpha)$ is a regular suborder of $P(\beta)$,
and $P(\alpha+1)$ is forcing equivalent to $P(\alpha) * \dot T$. 
In fact, if $G$ is a generic filter on $P(\alpha)$ and $T = \dot T^G$, then 
a branch $b$ through $T$ is $V[G]$-generic on $T$ iff the set 
$\{ (s,f) \in P(\alpha+1) / G : f(\alpha) \in b \}$ is $V[G]$-generic 
on $P(\alpha+1) / G$.

The next two lemmas complete the proof.

\begin{lemma}
	The poset $P(\omega_2)$ forces that $\dot T$ is not proper.
\end{lemma}

\begin{proof}
	Fix a large enough regular cardinal $\lambda$.
	Let $G$ be a generic filter on $P(\omega_2)$.
	In $V[G]$, let $\mathcal X$ denote the collection of all countable $M \prec H(\lambda)$
	satisfying: (a) $M = N[G]$ for some countable $N \prec H(\lambda)^V$,
	(b) there exists a master condition $q_N = (s_N,f_N) \in G$ for $N$, and
	(c) for all
	$a \in T$ of height $N \cap \omega_1$, there exists some
	$\gamma \in \dom(f_N) \cap N$ such that
	$f_N(\gamma) = a$.

	Working in $V$,
	for any countable elementary submodel $N \prec H(\lambda)^V$ and
	$p \in N \cap P(\omega_2)$, it is straightforward to build a master condition $q \le p$
	for $N$ forcing that property (c) holds for $\dot T$ in place of $T$. 
	It follows that in $V[G]$, $\mathcal X$ is stationary in $[H(\lambda)]^\omega$.
	Let $M = N[G] \in \mathcal X$ and we claim that there does not exist
	a master condition in $T$ for $M$.
	Note that an element $a$ of $T$ is a master condition for $N[G]$ iff for any dense open
	set $D \subset T$ in $N$, there is some $x <_T a$ such that $x \in D \cap N$.
	Suppose for a contradiction that $a$ is such an element.
	By dropping $a$ down to the height $N \cap \omega_1$ if necessary, assume
	without loss of generality that $a$ has height $N \cap \omega_1$.
	By the choice of $q_N$, fix $\gamma \in \dom(f_N) \cap N$ such that
	$f_N(\gamma) = a$.
	Since $q_N \in G$, it follows that $a \in b_\gamma$.
	
	Let $\dot a$ be a name which is forced by some $r \le q_N$ in $G$
	to be a master condition in $\dot T$ for $N[G]$, has height $N \cap \omega_1$,
	and is in $\dot b_\gamma$. 	
	Then $r * \dot a$ is a master condition in $P(\omega_2) * \dot T$ for $N$.
	But by property (4) of $P(\omega_2+1)$ the generic branches at coordinates
	$\gamma$ and $\omega_2$ must diverge.
	Since the dense set described in property (4) is in $N$ and $P(\omega_2+1)$ and
	$P(\omega_2) * \dot T$ are forcing equivalent, the fact that
	$r * \dot a$ is a master condition for $N$ implies that the branches with
	coordinates $\gamma$ and $\omega_2$ must
	diverge below height $N \cap \omega_1$.
	Hence, $r * \dot a$ forces that $a$ is not in $\dot b_\gamma$, which contradicts
	that $a \in b_\gamma$.
\end{proof}

\begin{lemma}
	The poset $P(\omega_2)$ forces that $\dot T$ preserves stationary subsets of
	$\omega_1$.
\end{lemma}

\begin{proof}
	Suppose for a contradiction that for some generic filter $G$ on $P(\omega_2)$,
	in $V[G]$ there exists	a stationary set $S \subset \omega_1$,
	a $T$-name $\dot C$ for a club subset
	of $\omega_1$, and some $x \in T$ such that
	$x \Vdash^{V[G]}_T \check S \cap \dot C = \emptyset$.
	Since $P(\omega_2)$ is $\omega_2$-c.c., using nice names and a density argument
	we can find some $\gamma < \omega_2$ such that
	$S = \dot S^{G \restriction \gamma}$ for some
	$P(\gamma)$-name $\dot S$, the $T$-name $\dot C$ is in $V[G \restriction \gamma]$,
	and there is a condition $(t,g) \in G$ such that $x \le_t g(\gamma)$.
	
	Now the forcing $P(\omega_2)$ can be factored as
	$P(\gamma) * \dot T * P(\omega_2) / \dot G_{\gamma+1}$, and we can
	correspondingly write $V[G] = V[G \restriction \gamma][b_\gamma][H]$ for some $H$,
	where $x \in b_\gamma$ since $(t,g) \in G$.
	By an absoluteness argument, $x \Vdash^{V[G_\gamma]}_T \check S \cap \dot C = \emptyset$.
	Hence, $S$ is disjoint from
	the club $\dot C^{b_\gamma}$ in $V[G_{\gamma+1}]$. 
	But then $\dot C^{b_\gamma} \in V[G]$, which contradicts that $S$
	is stationary in $V[G]$.
\end{proof}

\section{Proving the Main Theorem}

The forcing iteration described in Theorem 1.1 will involve forcing two types of posets:
(1) proper posets of size $\omega_1$, bookkeeping so that all such posets in the final model
will have been forced with $\omega_2$-many times in the iteration, and
(2) forcing notions which destroy posets which are stationary set preserving but not proper.

In order to prove that the forcing iteration is $\omega_2$-c.c., we will use a property
introduced by Shelah (see Chapter VIII Section 2 of \cite{PROPER}).

\begin{definition}
A poset $R$ satisfies the
\emph{$\omega_2$-properness isomorphism condition} (\emph{$\omega_2$-p.i.c.}\ for short) if and only if for every large enough regular cardinal $\theta$, for every well-ordering $<$ of $H_{\theta}$ and for all ordinals $\alpha < \beta < \omega_2$ the following holds: if $N_{\alpha}$ and $N_{\beta}$ are countable elementary submodels of $( H_{\theta}, \in , <, R)$
such that $\alpha \in N_{\alpha}$, $\beta \in N_{\beta}$, $N_{\alpha} \cap \omega_2 \subset \beta$, $N_{\alpha} \cap \alpha =N_{\beta} \cap \beta$, $p \in N_{\alpha} \cap R$ and $\pi : N_{\alpha} \to N_{\beta}$ is an isomorphism satisfying $\pi(\alpha)=\beta$ and $\pi  \restriction( N_{\alpha} \cap N_\beta)= id$, then there exists a master condition $q$ for $N_\alpha$,
extending $p$ and $\pi(p)$, such that
$$q \Vdash_R \pi`` (\dot{G} \cap \check N_\alpha) = \dot G \cap \check N_\beta.$$
\end{definition}

Every proper poset of size $\omega_1$ has the $\omega_2$-p.i.c., and if \textsf{CH} holds, then every $\omega_2$-p.i.c.\ poset satisfies the $\omega_2$-chain condition. Moreover, by Lemma 2.4 of Chapter VIII of \cite{PROPER}, under the assumption of \textsf{CH}, if $P_{\omega_2}$ is a countable support forcing iteration of $\{\dot{Q}_\beta \colon \beta< \omega_2\}$ such that every that every $\dot{Q}_\beta$ has the $\omega_2$-p.i.c.\ in $V^{P_{\omega_2}\restriction \beta}$, then $P_{\omega_2}$ has the $\omega_2$-chain condition. Therefore, \textsf{CH} implies that $P_{\omega_2}$ does not collapse cardinals. In the context of our specific iteration, we will apply this result by taking each $\dot{Q}_\beta$ to be either a name for a proper poset of size $\omega_1$ or a name for a poset $Q$ as described in the next theorem.

\begin{theorem}\label{posetQ}
There exists a proper countably distributive
poset $Q$ of cardinality $2^{\omega_1}$ with the $\omega_2$-p.i.c.\ satisfying that for every poset $P$ of cardinality $\omega_1$, if $P$ is not proper, then
$$\Vdash_Q \check{P} \mbox{ does not preserve stationary subsets of }\omega_1.$$
\end{theorem}

With this new ingredient, and assuming $\textsf{CH}$ together with $2^{\omega_1}=\omega_2$, the construction of a countable support forcing iteration 
$P_{\omega_2}$ witnessing Theorem \ref{mainthm-0-intro} is very natural. Since $2^{\omega_1}=\omega_2$, we can fix a function $\Phi:\omega_2\to H_{\omega_2}$ with the property that $\{\beta \in \omega_2\,:\,\Phi(\beta)=x\}$ is unbounded in $\omega_2$ for each $x\in H_{\omega_2}$. At stage $\beta< \omega_2$, if $\Phi(\beta)$ is a
$P_{\omega_2}\restriction \beta$-name for a proper poset of cardinality $\omega_1$, then let $\dot{Q}_\beta=\Phi(\beta)$. Otherwise, let $\dot{Q}_\beta$ be a $P_{\omega_2}\restriction \beta$-name for a poset $Q$ as in Theorem \ref{posetQ}.

We claim that $P_{\omega_2}$  forces $\textsf{MM}(\omega_1)$. 
Namely, for every $P_{\omega_2}$-name $\dot{\mathcal{P}}$ for a poset of cardinality  $\omega_1$ and for every sequence $(\dot{D}_i)_{i<\omega_1}$ of $\mathcal{P}_{\omega_2}$-names for dense subsets of $\dot{\mathcal{P}}$, there is a high enough $\beta<\omega_2$ such that $\dot{\mathcal{P}}$ and all members of $(\dot{D}_i)_{i<\omega_1}$ are $P_{\omega_2}\restriction \beta$-names and $\Phi (\beta)= \dot{\mathcal{P}}$. So, $\dot{Q}_\beta$ is either $\dot{\mathcal{P}}$ or a $P_{\omega_2}\restriction \beta$-name for a poset $Q$ as in Theorem \ref{posetQ} depending on whether or not $\dot{\mathcal{P}}$ is a
$P_{\omega_2}\restriction \beta$-name for a proper poset. This is possible thanks to the $\omega_2$-chain condition of $P_{\omega_2}$, the fact that $P_{\omega_2}$ has cardinality
$2^{\omega_1} = \omega_2$, and the unboundedness assumption on the bookkeeping function $\Phi$.

The rest of this section is devoted to proving Theorem \ref{posetQ}.

\begin{definition}
In an $\omega_1$-preserving forcing extension $V[G]$, a \emph{continuous $V$-reflection sequence} is a sequence $\langle\bar M_\ga\colon\ga\in C\rangle$ such that:
\begin{enumerate}
\item $C\subset\gw_1$ is a closed unbounded set;
\item for each $\ga\in C$, $\bar M_\ga$ is the transitive collapse of some elementary submodel of $( H_{\omega_2}^V, \in)$ such that $\ga=\gw_1^{\bar M_\ga}$;
\item (\emph{continuity}) for every $\ga\in C$ and every function $x\colon \ga^{<\gw}\to\ga$ in the model $\bar M_\ga$ there is $\gg\in\ga$ such that for every ordinal $\gd\in C$ between $\gg$ and $\ga$, $x\restriction\gd^{<\gw}\in\bar M_\gd$ (which implies by (2) above that $\gd$ is closed under $x$);
\item (\emph{reflection}) for every stationary set $S\subset [H_{\omega_2}^V]^{\omega}$ in $V$,
the set $\{\ga\in C\colon\bar M_\ga$ is the transitive collapse of some element of $S\}\subset\gw_1$ is stationary.
\end{enumerate}
\end{definition}

\begin{proposition}\label{reflection}
Let $V[G]$ be an $\omega_1$-preserving forcing extension in which there exists a continuous $V$-reflection sequence. In $V$, let $P$ be a forcing of cardinality $\omega_1$ which is not proper. Then $V[G]\models P$ does not preserve stationary subsets of $\gw_1$.
\end{proposition}

\begin{proof}
First, move to $V[G]$ and fix a continuous $V$-reflection sequence $\langle\bar M_\ga\colon\ga\in C\rangle$.

\begin{claim}
For every function $x\colon \gw_1^{<\gw}\to\gw_1$ in $V$, for all but countably many $\ga\in C$, $x\restriction \ga^{<\gw}\in\bar M_\ga$ holds.
\end{claim}

\begin{proof}
By the reflection property of the sequence, the set $S=\{\ga\in C\colon x\restriction\ga^{<\gw}\in\bar M_\ga\}\subset\gw_1$ is stationary. Use the continuity property of the sequence to find a regressive function $f\colon S\to\gw_1$ such that for every ordinal $\ga\in S$ and every ordinal $\gd\in C$ between $f(\ga)$ and $\ga$, $x\restriction\gd^{<\gw}\in\bar M_\gd$. Use Fodor's lemma to find an ordinal $\gg\in\gw_1$ such that the set $\{\ga\in S\colon f(\ga)=\gg\}$ is stationary. It is immediate from the definitions that for every ordinal $\gd\in C$ greater than $\gg$, $x\restriction\gd^{<\gw}\in\bar M_\gd$.
\end{proof}

Now, let $P$ be a poset of cardinality $\omega_1$ which is not proper, which we may assume without loss of generality has underlying set $\omega_1$. If $P$ collapses $\omega_1$, then it also
collapses $\omega_1$ in $V[G]$, and hence is not stationary set preserving.
So assume that $P$ preserves $\omega_1$.
Note that $P$ can be coded in $( H_{\omega_2}^V, \in)$ by a function $x\colon \gw_1^{<\gw}\to\gw_1$ in $V$ (for example, by the characteristic function of its partial ordering). By the claim, thinning out $C$ to a closed unbounded subset if necessary, we may assume that for every $\ga\in C$, $P\restriction\ga\in\bar M_\ga$.

Now, return to $V$ and observe that since $P$ is not proper, by a pigeonhole argument
there must be a condition $p\in P$ and a stationary set $S\subset [H_{\omega_2}]^{\omega}$ such that no model in $S$ has a master condition below $p$. Move to $V[G]$ and use the reflection property of the sequence to conclude that the set $T=\{\ga\in C\colon \bar M_\ga$ is the transitive collapse of some model in $S\}$ is stationary. It will be enough to show that in $V[G]$, $p\Vdash_P\check T$ is nonstationary.

To this end, let $\dot E$ be the $P$-name for the set $\{\ga\in C\colon$ the $P$-generic filter has nonempty intersection with every maximal antichain of $P\restriction\ga$ in the model $\bar M_\ga\}$.

\begin{claim}
$\Vdash_P \dot E\subset\gw_1^V$ is a closed unbounded set.
\end{claim}

\begin{proof}
First, argue for the unboundedness. Let $q\in P$ be a condition and $\gg\in\gw_1$ be an ordinal. Back in $V$, consider the set $U\subset [H_{\omega_2}]^{\omega}$ of all models which contain $\gamma$ and have a master condition below $q$. Let us prove that the set $U$ is stationary.

To this end, let $f\colon H_{\omega_2}^{<\gw}\to H_{\omega_2}$ be a function. To find a model $M\in U$ closed under the function $f$, let $\theta$ be a large enough regular cardinal and let $X=\langle N_\ga\colon\ga\in\gw_1\rangle$ be a continuous increasing tower of countable elementary submodels of $H_\theta$ containing $f$ as an element, and let $N=\bigcup_\ga N_\ga$. Let $H\subset P$ be a generic filter containing $q$, and consider the models $N_\ga[H]$ for $\ga\in\gw_1$ and $N[H]$. Since the poset $P$ is a subset of $N$, it is clear that $N[H]\cap V=N$. The models $\langle N_\ga[H]\colon\ga\in\gw_1\rangle$ form a continuous increasing sequence of countable subsets of $N[H]$, so $Y=\langle N_\ga[H]\cap V\colon\ga\in\gw_1\rangle$ is a continuous increasing sequence of countable subsets of $N[H]\cap V=N$. Since $\gw_1$ is preserved passing to $V[H]$, the sequences $X$ and $Y$ must intersect at some point, i.e.\ there must be an ordinal $\ga\in\gw_1$ such that $N_\ga[H]\cap V=N_\ga$. Fix a condition $r\leq q$ in $H$ such that $r\Vdash^V_P N_\ga[\dot H]\cap V=N_\ga$. Then $r$ is a master condition for $N_\ga\cap H_{\omega_2}$, and $N_\ga\cap H_{\omega_2}$ is a model in the set $U$ closed under the function $f$.

 In $V[G]$ again, use the reflection property to find an ordinal $\ga\in C$ which is greater than $\gg$ and such that $\bar M_\ga$ is the transitive collapse of some model $M$ in $U$. By the definition of $U$, fix a master condition $r\leq q$ for $M$. It easily follows that
 $r$ forces that $\check\ga\in\dot E$. This completes the proof of the unboundedness of $\dot E$.

For the closure, suppose that some condition $q\in P$ forces an ordinal $\ga\in C$ to be a limit point of $\dot E$. To show that $q\Vdash\check\ga\in\dot E$, let $A\in\bar M_\ga$ be a maximal antichain of $P\restriction\ga$ in the model $\bar M_\ga$ and let $r \le q$; we must find a condition in $A\cap \ga$ compatible with $r$. To do this, apply the continuity property to a suitable function to find an ordinal $\gamma\in\ga$ such that for every ordinal $\gd\in C$ between $\gg$ and $\ga$, $A\cap\gd$ is a maximal antichain of $P\restriction\gd$ in the model $\bar M_\gd$. Since $r$ forces that $\ga$ is a limit point of $\dot E$,
find a condition $s\leq r$ and an ordinal $\gd\in C$ between $\gg$ and $\ga$ such that $s\Vdash\check\gd\in\dot E$. By the definition of the name $\dot E$, there must be an element of $A\cap\gd$ compatible with the condition $s$, and hence with $r$. This completes the proof.
\end{proof}

It is now clear from the definitions that in $V[G]$, $p\Vdash_P\dot E\cap\check T= \emptyset$.
The proof of the proposition is complete.
\end{proof}

We now define the poset for adding a continuous $V$-reflection sequence.

\begin{definition}
$Q$ is the set of all pairs $q=\langle a_q, b_q\rangle$ where

\begin{enumerate}
\item $a_q$ is a function whose domain is a closed countable subset of $\gw_1$ called the \emph{support of} $q$, $\supp(q)$;
\item for every ordinal $\ga\in\supp(q)$, writing $M=a_q(\ga)$, we have that $M$ is the transitive collapse of a countable elementary submodel of $( H_{\omega_2}^V, \in)$
such that $\gw_1^{M}=\ga$;
\item (continuity) for every $\ga\in\supp(q)$ and every function $x\colon \ga^{<\gw}\to\ga$ in the model $a_q(\ga)$, there is $\gg\in\ga$ such that for every ordinal $\gd\in\supp(q)$ between $\gg$ and $\ga$, $x\restriction\gd^{<\gw}\in a_q(\gd)$;
\item $b_q$ is a countable set of functions from $\gw_1^{<\gw}$ to $\gw_1$.
\end{enumerate}

\noindent The ordering is given by $r\leq q$ if $\supp(r)$ is an end-extension of $\supp(q)$, $a_q\subset a_r$, $b_q\subset b_r$, and for every ordinal $\ga\in\supp(r)\setminus\supp(q)$ and every function $x\in b_q$,  $x\restriction\ga^{<\gw}\in a_r(\ga)$.
\end{definition}

It is not difficult to see that the forcing $Q$ has cardinality $2^{\omega_1}$, and the relation $\leq$ is transitive and reflexive. 
The relevant forcing properties of $Q$ are all derived from the following consideration:

\begin{definition}
Let $M$ be a countable elementary submodel of $H(\kappa)$, for
a large enough regular cardinal $\kappa$. Let $g\subset M$ be a filter meeting all open dense subsets of $Q$ in $M$.

\begin{enumerate}
\item Let $a=\bigcup_{s\in g}a_s\cup\{\langle M\cap\gw_1, \bar M\rangle\}$, where $\bar M$ is the transitive collapse of the model $M\cap H_{\omega_2}$;
\item let $b$ be the set of all functions from $\gw_1^{<\gw}$ to $\gw_1$ belonging to the model $M$;
\item let $r(M, g)=\langle a, b\rangle$.
\end{enumerate}
\end{definition}

\begin{proposition}
$r(M, g)$ is a condition in $Q$ which is a common lower bound of all conditions in $g$.
\end{proposition}

\begin{proof}
 Write $\ga=M\cap\gw_1$. First of all, the fact that $g$ is a filter shows that $\bigcup_{s\in g}a_s$ is a function, and its domain $c$ is a subset of $M\cap\gw_1$ which is closed except perhaps at its supremum. A simple density argument shows that in fact $\sup(c)=\ga$. Thus, to verify that $r(M, g)$ is a condition, it is only necessary to check the continuity of $a$ at $\ga$. Let $y\colon\ga^{<\gw}\to\ga$ be any function in the model $\bar M$, and let $x\in M$ be the function whose collapse is $y$. By a density argument, there must be a condition $s\in g$ such that $x\in b_s$. The definition of the ordering on $Q$ then shows that the ordinal $\gg=\max\supp(s)$ witnesses the continuity condition for $\ga$ and $x$.

To verify that for every condition $s\in g$, $r(M, g)\leq s$ holds, it is enough to verify that for every ordinal $\gd\in\dom(a)\setminus\supp(s)$ and every $x\in b_s$, it is the case that $x\restriction\gd^{<\gw}\in a(\gd)$. For $\gd\in\ga$ this is immediately clear from the assumption that $g$ is a filter. If $\gd=\ga$, then $x\in M$ since $x\in b_s$ and $s\in M$; by the elementarity of $M$ we conclude again that $x\restriction\gd^{<\gw}$ belongs to $a(\ga)$, since it is the transitive collapse image of the function $x$.
\end{proof}

\begin{corollary}
The poset $Q$ satisfies the following properties:

\begin{enumerate}
\item proper;
\item countably distributive;
\item $\omega_2$-p.i.c.
\end{enumerate}
\end{corollary}

\begin{proof}
For (1), let $q\in Q$ be a condition and let $M$ be a countable elementary submodel of
$H(\kappa)$, for a large enough regular cardinal $\kappa$,
such that $q$ and $Q$ are in $M$.
Construct a filter $g\subset M\cap Q$ containing the condition $q$ and meeting all dense open subsets of $Q$ which belong to the model $M$. It is immediate that $r(M,g)$ is a master condition for the model $M$ below $q$.
For (2), if in addition $\{D_n\colon n\in\gw\}$ is a countable collection of open dense subsets of $Q$ and $M$ is selected in such a way that each $D_n$ is in M, then $r(M, g)$ is a condition below $q$ in the intersection $\bigcap_nD_n$.

For (3), suppose that $M, N$ are two isomorphic countable elementary submodels.
By the Mostowski collapse lemma, the isomorphism is unique, and we denote it by $\pi\colon M\to N$. Let $q\in M\cap Q$ be an arbitrary condition and let $g\subset M\cap Q$ be a filter having $q$ as an element and meeting all open dense subsets of $Q$ which belong to the model $M$. It will be enough to show that there is a condition $r$ extending all the elements of the set $g\cup \pi''g$. To find $r$, write $r(M, g)=\langle a_M, b_M\rangle$ and $r(N, \pi''g)=\langle a_N, b_N \rangle$, and observe that $a_M=a_N$ since the isomorphism $\pi$ fixes $M \cap \omega_1=N \cap \omega_1$ and because the two models $M$ and $N$ have the same transitive collapse. So, $r(M, g)$ and $r(N, \pi''g)$ are compatible as witnessed by the common extension $r=\langle a_M, b_M\cup b_N\rangle$ and $r$ works as desired.
\end{proof}

\begin{corollary}
Let $G\subset Q$ be a generic filter. In the model $V[G]$, let $F=\bigcup\{ a\colon\exists r\in G\ a=a_r\}$. Then $F$ is a continuous $V$-reflection sequence.
\end{corollary}

\begin{proof}
It is immediately clear that $\dom(F)$ is a closed unbounded subset of $\gw_1$ and that $F$ satisfies the continuity property. Thus, it will be enough to verify the reflection property. For this, return to the ground model, let $q\in Q$ and let $S\subset [H_{\omega_2}]^{\omega}$ be a stationary set. Let also $\dot E$ be a $Q$-name for a closed unbounded subset of $\gw_1$. It will be enough to find a condition $r\leq q$ and an ordinal $\ga\in\supp(r)$ such that $a_r(\ga)$ is the  transitive collapse of a model in $S$ and $r\Vdash\check\ga\in\dot E$.

To this end, use the stationarity of the set $S$ to find a countable elementary submodel $M$ of $H(\kappa)$ for some large enough regular cardinal $\kappa$, containing both $q$ and $\dot E$ such that $M\cap H_{\omega_2}\in S$. Find a filter $g\subset Q\cap M$ generic over $M$ containing the condition $q$, and let $r=r(M, g)$ and $\ga=M\cap\gw_1$. It is clear that $r\leq q$, $r\Vdash\check\ga\in\dot E$ since $r$ is a master condition for $M$, and $a_r(\ga)$ is a model isomorphic to $M\cap H_{\omega_2}\in S$.
\end{proof}

The proofs of Theorems \ref{posetQ} and \ref{mainthm-0-intro} are complete.

\section{Final Remarks}

We finish this article with a few open questions. The first of them is motivated by the size of the poset $Q$ used in Theorem \ref{posetQ} and the second one comes from the search for fragments of $\textsf{MM}$ implying (as $\textsf{PFA}$ does) that the continuum is equal to $\omega_2$.

\begin{question}
Does $\textsf{PFA}(\omega_1)$ imply $\textsf{MM}(\omega_1)$?.
\end{question}

\begin{question}
Does  $\textsf{MM}(\omega_1)$ imply $2^{\omega}=\omega_2$?.
\end{question}

\begin{question}
Does  $\textsf{MM}(\omega_1)$ imply that every stationary set preserving forcing of 
size $\omega_1$ is proper? Does this conclusion hold in the model constructed above?
\end{question}

In \cite{forcing-conseqs}, Asper\'o and Mota proved that the forcing axiom $\textsf{FA}(\Gamma, \omega_1)$, for $\Gamma$ being the class of all finitely proper posets of cardinality $\omega_1$, is consistent with the continuum being arbitrarily large. Very recently, Asper\'o and Golshani have improved that result by showing that $\textsf{PFA}(\omega_1)$ is also compatible with $2^{\omega}>\omega_2$ (see \cite{A-G}). Therefore, a positive answer for our first question would imply a negative answer for the second one.

\end{document}